%% file: a_counter_example_to_a_kroger_type_spectral_inequality.tex
\documentclass[a4paper,12pt,reqno]{amsart}	
\usepackage[centertags]{amsmath}		
\usepackage{amsfonts}		
\usepackage{amssymb}		
\usepackage{amsthm}			
\usepackage[utf8]{inputenc}
\usepackage[T1]{fontenc}	
\usepackage{lmodern}
\usepackage[english]{babel}
\usepackage[autolanguage]{numprint}
\usepackage{dsfont}
\usepackage{textcomp}
\usepackage[a4paper,margin=2cm,headheight=14.5pt]{geometry}	
\usepackage{graphicx}		
\usepackage{fancyhdr}		
\usepackage[colorlinks,linkcolor=blue,citecolor=blue]{hyperref}
\usepackage[hypcap=false]{caption}
\usepackage{mathrsfs}			
\usepackage{calc}

\theoremstyle{plain}
\newtheorem{thm}{Theorem}

\newtheorem{lem}[thm]{Lemma}
\newtheorem{prop}[thm]{Proposition}
\theoremstyle{definition}

\theoremstyle{remark}

\numberwithin{equation}{section}

\renewcommand{\textbf}[1]{\begingroup\bfseries\mathversion{bold}#1\endgroup} 
\newcommand{\grad}{\mathrm{grad}}
\newcommand{\dive}{\text{div}}
\newcommand{\vol}{\text{Vol}}
\newcommand{\B}{\mathbb{B}}
\newcommand{\N}{\mathbb{N}}

\newcommand{\R}{\mathbb{R}}

\renewcommand{\S}{\mathbb{S}}
\newcommand{\Sii}{\tilde{\Sigma}}
\title{Counter-example to a Kröger type spectral inequality}
\author{Luc Pétiard}
\date{}
\begin{document}


\address{Université de Neuchâtel, Institut de Mathématiques, Neuchâtel, Switzerland}
\email{luc.petiard@unine.ch; luc.petiard@laposte.net}

\begin{abstract}

Given a Riemannian manifold, Weyl's law indicates how the spectrum of the Laplacian behaves asymptotically. Because of that result, there has been a growing interest in finding geometrical bounds compatible with this law. In the case of hypersurfaces, the isoperimetric ratio is a natural geometrical quantity, that allows to bound the spectrum from above. We investigate the problem and find an example of hypersurface where the eigenvalues are minorated by the isoperimetric ratio.

\end{abstract}

\maketitle

\section{Introduction}
\addcontentsline{toc}{chapter}{Introduction}

Throughout this article we will consider smooth and compact hypersurfaces of $\R^{n+1}$, namely submanifolds of dimension $n$ in $\R^{n+1}$ equipped with induced metric.
The associated spectrum of the Laplace operator $\Delta = -\dive(\grad)$, is discrete, positive, and denoted
\[
0 = \lambda_0 < \lambda_1 \leqslant \lambda_2 \leqslant \dots \nearrow +\infty
\]

 One has very few examples of manifolds $M$ whose spectra are known, and an accurate estimation of the spectrum is difficult, even for the first non-zero eigenvalue. However we can still recover some information such as the volume and the dimension of the hypersurface when $k$ goes to infinity, according to Weyl's law (see \cite{panoramic_view}, p. 108). It states that there exists a constant $W(n)=\frac{(2\pi)^2}{\omega_n^{\frac{2}{n}}}$ such that for all fixed compact hypersurfaces $\Sigma$ of $\R^{n+1}$ one has:
\[
\lambda_k(\Sigma) \underset{k \to +\infty}{\sim}
W(n)\left(\dfrac{k}{\vol(\Sigma)}\right)^{\frac{2}{n}}
\]
where $\omega_n$ is the volume of the unit ball in $\R^n$ and $\vol(\Sigma)$ is the volume of $\Sigma$. This asymptotic law is also true for the eigenvalues $\mu_k(\Omega)$ of the laplacian on a bounded domain $\Omega$ of $\R^{n}$ with Lipschitz boundary, and Dirichlet (respectively Neumann) condition on it. Furthermore, the Pólya conjecture states that $\mu_k(\Omega)$ is, for all $k \in \N$, bounded from above (respectively from below) by $W(n)\left(\frac{k}{\vol(\Omega)}\right)^{\frac{2}{n}}$. Kröger \cite{kroger_polya} has proved that the Pólya conjecture is true for Neumann conditions, up to a coefficient $\left(\frac{n+2}{n}\right)^{\frac{2}{n}}$.

However, B. Colbois, E. Dryden et A. El Soufi showed that a result of Kröger type is not possible in the context of compact hypersurfaces. More precisely they showed in \cite{colbois_bounding_eigenvalues_operator} that if $\Sigma$ is an hypersurface of $\R^{n+1}$ with $n\geqslant 3$,

\begin{equation}
\label{sup_lambda_1=infini}
\underset{X}{\sup}\ \lambda_1 (X(\Sigma)) \vol(X(\Sigma))^{\frac{2}{n}} = \infty
\end{equation}

where the supremum is taken over the set of embeddings from $\Sigma$ to $\R^{n+1}$.

In dimension $2$ we have the result:
\[
\underset{\Sigma}{\sup}\ \lambda_1 (\Sigma) \vol(\Sigma)^{\frac{2}{n}} = \infty
\]
where the supremum is taken over the set of compact surfaces $\Sigma \subset \R^3$.

\medskip
Therefore we have to impose geometric restrictions in order to bound the spectrum from above. For example in the same paper the authors proved that if $\Sigma$ is a convex hypersurface of dimension $n$, then 
\[
\lambda_k(\Sigma)\vol(\Sigma)^{\frac{2}{n}}
\leqslant c(n)k^{\frac{2}{n}}
\]
where $c(n)$ is a constant depending only on the dimension.

For compact surfaces $\Sigma$ of $\R^3$, a natural restriction is the genus. A. Hassannezhad \cite{asma_conformal_upper_bounds_laplace_steklov_problem} showed the existence of two constants $C_1$ and $C_2$ such that for all compact surfaces $\Sigma$ of $\R^3$ we have:
\[
\lambda_k(\Sigma)\vol(\Sigma) \leqslant C_1\text{genus}(\Sigma)+C_2k.
\]
In particular for all $k \geqslant \text{genus}(\Sigma)$, we get
\[
\lambda_k(\Sigma)\vol(\Sigma) \leqslant C_3k
\]
where $C_3 = C_1+C_2$ is a universal constant.
In other words, separating the geometric term from the asymptotic one gives an inequality ''à la Kröger'', but only for eigenvalues $\lambda_k(\Sigma)$ such that $k\geqslant \text{genus}(\Sigma)$. This was motivated by an older result of N. Korevaar \cite{korevaar_upper_bounds_eigenvalues_conformal_metrics} who had shown the existence of a constant $C$ such that
\[
\lambda_k(\Sigma)\vol(\Sigma) \leqslant C(\text{genus}(\Sigma)+1)k.
\]
Hassannezhad also improved another result of Korevaar on conformal class in dimension $n \geqslant 3$.

Let us now focus on the main result of the article \cite{colbois_isoperimetric_control_spectrum_hypersurface} from B. Colbois, A. El Soufi and A. Girouard in which they link the eigenvalues of a compact hypersurface to the isoperimetric ratio:

\begin{thm}
\label{thmcolbois}

There exists an explicit constant $\gamma(n)$ such that, for all bounded domains $\Omega$ of $\R^{n+1}$ with a $\mathscr{C}^2$-boundary $\Sigma = \partial\Omega$, and for all $k\geqslant 1$,
\[
\lambda_k(\Sigma)\cdot\vol(\Sigma)^{\frac{2}{n}} \leqslant \gamma(n) I(\Sigma)^{1+\frac{2}{n}} k^{\frac{2}{n}}
\]

with $I(\Sigma)= \dfrac{\vol_n(\Sigma)}{\vol_{n+1}(\Omega)^{\frac{n}{n+1}}}$ the isoperimetric ratio and $\gamma(n) = \frac{2^{10n+18+8/n}}{n+1}\omega_{n+1}^{\frac{1}{n+1}}$.
\end{thm}

\bigskip
Thus the question is whether there is a Kröger type inequality for large $k$ in dimension $n\geqslant 3$, that is  if there exist a constant $A>0$ and a continuous function $f:(0,+\infty)\mapsto \R$ such that for all hypersurfaces $\Sigma$ and for all $k$,
\begin{equation}
\label{inegseparee}
\lambda_k(\Sigma)\vol(\Sigma)^{\frac{2}{n}} \leqslant f(I(\Sigma)) + Ak^{\frac{2}{n}}.
\end{equation}

\medskip
We will show that this is not possible by constructing a family of counter-examples, that is:
\begin{thm}\ %
\label{theoreme_connexe}

Let $n\geqslant3$ and let $f:\left(0,+\infty \right)\to \R$ be a continuous function. For all $A > 0$ there exists $k_0\in \N$ such that for all $k\geqslant k_0$, there exists a connected hypersurface $\Sigma \subset \R^{n+1}$ verifying
\[
\lambda_k(\Sigma)\vol(\Sigma)^{\frac{2}{n}} > f(I(\Sigma)) + Ak^{\frac{2}{n}}.
\]
Moreover the hypersurface $\Sigma$ can be chosen diffeomorphic to a given hypersurface.
\end{thm}

\medskip

The idea of the construction is as follows and will be done in sections \ref{section_non_connected} and \ref{section_tubular_attachment}.
First we take a submanifold of $\R^{n+1}$ and modify it to get its first eigenvalue as large as desired. Then we consider the disjoint union of this hypersurface with the sphere $\S^n$, whose spectrum is well-known. This way we should be able to express the spectrum of the union with the one of the sphere.

We then ''glue'' them with a tube, thin enough not to alter the behaviour of the spectrum, and making the union diffeomorphic to the first submanifold.

In the end we obtain a connected hypersurface having an important property; the beginning of its spectrum is the one of the sphere, but its area is very large compared to the one of the sphere. We thus show that even with a controlled isoperimetric ratio, it is possible to find a hypersurface having a spectrum sufficiently big to contradict the inequality \ref{inegseparee}.
The first chapter covers the tools required for this construction.

\medskip

Let us remark that trying to find whether or not one can ''separate'' the geometric term from the asymptotic one, is not immediate. For example, let us call $i(\Sigma)$ the intersection index, that is, for a hypersurface, the maximal number of points you can get on a line intersecting $\Sigma$. Let us have a look at the following result\cite{colbois_bounding_eigenvalues_operator}
\[
\lambda_k(\Sigma)\vol(\Sigma)^{\frac{2}{n}}\leqslant B(n)i(\Sigma)^{1+\frac{2}{n}}k^{\frac{2}{n}}
\]
where $B(n)$ is an explicit constant depending only the dimension $n$.
Then it is still an open problem to know whether or not we can bound the spectrum in the following way :
\[
\lambda_k(\Sigma)\vol(\Sigma)^{\frac{2}{n}}
\leqslant B_1(n)F(i(\Sigma)) + B_2(n)k^{\frac{2}{n}}
\]
where $B_1(n)$ and $B_2(n)$ are two constants depending only on $n$ and $F$ is a continuous function.

\section{Non-connected case}
\label{section_non_connected}

In this first part we will prove the following result:
\begin{thm}\ %
\label{theoreme_cas_non_connexe}

Let $n\geqslant3$ and let $f$ be a continuous function, $f:\left(0,+\infty \right)\to \R$. For all $A > 0$ there exists $k_0\in \N$ such that for all $k\geqslant k_0$, there exists a smooth and compact non-connected hypersurface $\Sii_k \subset \R^{n+1}$ verifying
\[
\lambda_k(\Sii_k)\vol(\Sii_k)^{\frac{2}{n}} > f(I(\Sii_k)) + Ak^{\frac{2}{n}}.
\]
\end{thm}

\bigskip
Recall result \ref{sup_lambda_1=infini} and let $M$ be a manifold of dimension $n$ which can be embedded in $\R^{n+1}$. Then there exists a sequence of embedding $X$ such that $\lambda_1 (X(M)) \vol(X(M))$ is as large as desired. If we call $\Sigma = X(M)$ the image of this embedding, it is equivalent to say:
\begin{equation}
\label{lambda_1_grand}
\forall L > 0, \exists \Sigma \text{ s.t. } \lambda_1(\Sigma)\vol(\Sigma)^{\frac{2}{n}} \geqslant L.
\end{equation}

Then for all $k$, there exists $\Sigma_k = X_k(\Sigma)$ such that
\[
\lambda_1(\Sigma_k)\vol(\Sigma_k)^{\frac{2}{n}}
\geqslant 4A\left(\lambda_k(\S^n)+1\right)\dfrac{\vol(\S^n)^{\frac{2}{n}}}{W(n)}\cdotp
\]
As $\lambda_1(\Sigma_k)\vol(\Sigma_k)^{\frac{2}{n}}$ is a homothetic invariant quantity, we can assume that
\[
\label{choix_lambda_1_et_volume}
\vol(\Sigma_k)^{\frac{2}{n}}=4A\dfrac{\vol(\S^n)^{\frac{2}{n}}}{W(n)} \ \ \ \ \ \ \ \ \ \ \ 
\lambda_1(\Sigma_k) \geqslant \lambda_k(\S^n)+1
\]

We will denote by $\Omega_k$ the domain of dimension $n+1$ such that $\partial \Omega_{k} = \Sigma_{k}$.
The notation ''$M\sqcup N$'' will designate the disjoint union of two manifolds $M$ and $N$. To this union we can associate its spectrum, constituted of the union of the two spectra of $M$ and $N$.
The notation $\Sii_k = \Sigma_k \sqcup \S^n$ will refer to the union of $\Sigma_k$ with the sphere $\S^n$.

\begin{lem}\ %
	
We have:
\begin{equation}
\label{egalite_sphere_union}
\left\{ \begin{array}{ll}
\lambda_0(\Sii_k)=\lambda_1(\Sii_k)=0& \\
 \lambda_j(\Sii_k)=\lambda_{j-1}(\S^n)\ \forall j=2\dots k &  \end{array} \right.
\end{equation}
\end{lem}

\begin{proof}
We know that $\text{Sp}(\Sigma_k \sqcup \S^n)=\text{Sp}(\Sigma_k) \cup \text{Sp}(\S^n)$ and also that $\lambda_1(\Sigma_k) \geqslant \lambda_k(\S^n)+1$. As a consequence the beginning of the spectrum of the union  of $\S^n$ with $\Sigma_k$ can be represented as follows:
\begin{center}
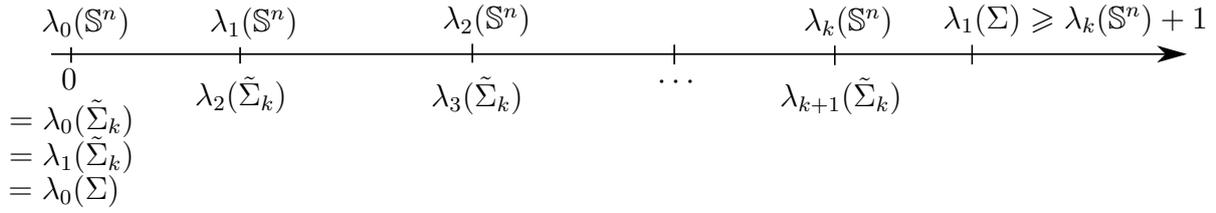

\def\svgscale{0.8}
\input{img_spectre_decale_bis.tex}
\captionof{figure}{The eigenvalues are shifted}
\label{sp_decale}
\end{center}
\end{proof}

We can now prove the following lemma:
\begin{lem}
There exists $k_1 > 0$ such that $\forall k\geqslant k_1$ we have
\[
\lambda_k(\Sii_k)\vol(\Sii_k)^{\frac{2}{n}}
> 2Ak^{\frac{2}{n}}.
\]
\end{lem}

\begin{proof}
	By the choice of the volume of $\Sigma_k$ as above, we get:
\[
\vol(\Sii_k)^{\frac{2}{n}}
\geqslant \vol(\Sigma_k)^{\frac{2}{n}}
= 4A\dfrac{\vol(\S^n)^{\frac{2}{n}}}{W(n)}\cdot
\]
Therefore
\[
\dfrac{\lambda_k(\Sii_k)\vol(\Sii_k)^{\frac{2}{n}}}{k^{\frac{2}{n}}}
= \dfrac{\lambda_{k-1}(\S^n)\vol(\Sii_k)^{\frac{2}{n}}}{k^{\frac{2}{n}}}
\geqslant 4A\dfrac{\lambda_{k-1}(\S^n)\vol(\S^n)^{\frac{2}{n}}}{W(n)k^{\frac{2}{n}}}\cdotp
\]
But we know Weyl's asymptotic law for the spectrum of $\S^n$ gives:
\[
\underset{k \to \infty}{\lim}
\dfrac{\lambda_{k-1}(\S^n)\vol(\S^n)^{\frac{2}{n}}}{W(n)k^{\frac{2}{n}}} = 1.
\]
So there exists $k_1 > 0$ such that $\forall k \geqslant k_1$, 
\[
\dfrac{\lambda_{k-1}(\S^n)\vol(\S^n)^{\frac{2}{n}}}{W(n)k^{\frac{2}{n}}} > \dfrac{1}{2}
\]
so it is quite clear that for $k \geqslant k_1$,
\[
\lambda_k(\Sii_k)\vol(\Sii_k)^{\frac{2}{n}}
> \dfrac{1}{2} 4Ak^{\frac{2}{n}}
=2Ak^{\frac{2}{n}}.
\]
\end{proof}

\begin{lem}
The isoperimetric ratio of $\Sii_k$ is bounded as follows:
\[
I(\S^n)
\leqslant I(\Sii_k)
\leqslant (n+1)^{\frac{n}{n+1}} \left(\left(\dfrac{4A}{W(n)}\right)^{\frac{n}{2}}+1\right)\vol(\S^n)^{\frac{1}{n+1}}
\]
Note that the bounds do not depend on $k$.
\end{lem}

\begin{proof}\ %

The left-hand side is simply the isoperimetric inequality in $\R^{n+1}$.
As for the right-hand side, let us call $\Omega_k$ the interior of $\Sigma_k$. Then the interior of $\Sii_k$ is $\Omega_k \sqcup \B^{n+1}$.

Consequently $\vol(\Omega_k \sqcup \B^{n+1}) \geqslant \vol(\B^{n+1}) = \frac{1}{n+1}\vol(\S^n)$

and $\vol(\Sii_k) = \vol(\Sigma_k) + \vol(\S^n)= \left(\frac{4A}{W(n)}\right)^{\frac{n}{2}}\vol(\S^n) + \vol(\S^n) = \left(\left(\frac{4A}{W(n)}\right)^{\frac{n}{2}}+1\right)\vol(\S^n)$.

Then
\begin{align*}
 I(\Sii_k) & = \dfrac{\vol(\Sii_k)}{\vol(\Omega_k \sqcup \B^{n+1})^{\frac{n}{n+1}}}
 	= \dfrac{\left(\left(\frac{4A}{W(n)}\right)^{\frac{n}{2}}+1\right)\vol(\S^n)}{\vol(\Omega_k \sqcup \B^{n+1})^{\frac{n}{n+1}}}\\
   & \leqslant \dfrac{\left(\left(\frac{4A}{W(n)}\right)^{\frac{n}{2}}+1\right)\vol(\S^n)}{\left(\frac{1}{n+1}\vol(\S^n)\right)^{\frac{n}{n+1}}}
    = (n+1)^{\frac{n}{n+1}} \left(\left(\frac{4A}{W(n)}\right)+1\right)^{\frac{n}{2}}\vol(\S^n)^{\frac{1}{n+1}}.
\end{align*}

\end{proof}

Now we can prove Theorem \ref{theoreme_cas_non_connexe} for non-connected hypersurfaces.

\begin{proof}\ %

The isoperimetric ratio $I(\Sii_k)$ belongs to a closed and bounded interval for all $k$. So for all continuous functions $f:\left( 0,+\infty \right)\to \R$, we know that $f(I(\Sii_k))$ also belongs to a closed and bounded interval. In particular there exists $k_2 > 0$ such that for all $k \geqslant k_2, f(I(\Sii_k)) < Ak^{\frac{2}{n}}$.

If we take $k_0 = \max(k_1,k_2)$, we thus have $\forall k \geqslant k_0$,
\[
\lambda_k(\Sii_k)\vol(\Sii_k)^{\frac{2}{n}}
> 2Ak^{\frac{2}{n}}
> f(I(\Sii_k)) + Ak^{\frac{2}{n}}.
\]
\end{proof}

\section{Tubular attachment}
\label{section_tubular_attachment}

We just saw an example where the hypersurface was disconnected. We can obtain the same result for a connected hypersurface, simply by attaching the two components of $\bar \Sigma_k$. This is possible while keeping control on the spectrum and on the isoperimetric ratio, and we are going to show that in the following proposition. The construction of $\Sigma$ being simple but technical, the reader may refer to the figure \ref{gluing} for a visual intuition.

\begin{prop}\ %

Let $\Sigma_1$ and $\Sigma_2$ be two compact and connected hypersurfaces of $\R^{n+1}$. For all $\epsilon > 0$ and for all $N \in \N^*$, there exists a hypersurface $\Sigma = \Sigma_{\epsilon,N}$ such that
\[
\left|I(\Sigma_{\epsilon,N})-I(\Sigma_1 \sqcup \Sigma_2)\right| < \epsilon
\]
and $\forall k \leqslant N$
\[
\left|\lambda_k(\Sigma_{\epsilon,N}) - \lambda_k(\Sigma_1 \sqcup \Sigma_2)\right|<\epsilon.
\]
\end{prop}

\begin{proof}\ %

Let $h>0$. Without loss of generality we can assume that $\Sigma_1 \subset \{x_{n+1}\geqslant h\}$ and $\Sigma_2 \subset \{x_{n+1}\leqslant -h\}$ and that the points $(0,\dots,0,h)$ and $(0,\dots,0,-h)$ belong to $\Sigma_1$ and $\Sigma_2$ respectively. We note $p=(0,\dots,0,h)$ and $q=(0,\dots,0,-h)$.

A hypersurface can be locally represented by a graph. As a consequence, for all $\delta > 0$ small enough, there exists a function  $f:\B^n(0,4\delta) \subset \R^n \longrightarrow \R$ such that the set $V_1 = \{(x,f(x)), x\in \B^n(0,4\delta)\}$ is an open set of $p$ in $\Sigma_1$.

Note that it naturally implies $f(0,\dots,0) = h$ and $\forall x\in \B^n(0,4\delta), f(x) \geqslant h$.

Now let $f_1 \in \mathscr{C}^{\infty}(\B^n(0,4\delta),\R)$ such that $f_1 \geqslant h$ and
\[
f_1(x) = \left\{
    \begin{array}{ll}
        f(x) & \text{if } x \in \B^n(0,4\delta)\setminus \B^n(0,2\delta)\\
        h & \text{if } x \in \B^n(0,\delta)
    \end{array}
\right.
\]
Let us call $V_1^{\delta} = \{(x,f_1(x)), x\in \B^n(0,4\delta)\}$ the graph of this new function $f_1$. It will generate a new hypersurface $\Sigma_1^{\delta}$ equal to $\left(\Sigma_1\setminus V_1\right) \cup V_1^{\delta}$, flat on the open set $\B^n(0,\delta)\times \{h\}$.

\medskip
We modify $\Sigma_2$ around $q$ on an open set $V_2$ in the same way, in order to create an open set $V_2^{\delta}$. Thus we form a hypersurface $\Sigma_2^{\delta}=\left(\Sigma_2\setminus V_2\right) \cup V_2^{\delta}$, flat on the open set $\B^n(0,\delta)\times \{-h\}$.

For the following we will use the notation $B_1^{\delta} = \B^n(0,\delta)\times \{h\} \subset \Sigma_1^{\delta}$ and $B_2^{\delta} = \B^n(0,\delta)\times \{-h\}\subset \Sigma_2^{\delta}$. We now just have to glue the two parts.

For this purpose we call  $R^{\delta,h}$ the surface of revolution defined by
\[
R^{\delta,h}=\left\{(x_1,\dots,x_{n+1}) : x_1^2+\dots+x_n^2=\varphi^2(x_{n+1}) \text{ et } -h\leqslant x_{n+1} \leqslant h\right\}
\]
where $\varphi : (-h,h) \longrightarrow \R$ is an even, increasing function, bounded from above by $\delta$ on $[0,h]$ and such that $\varphi\left(\left\lbrack -\frac{h}{2},\frac{h}{2} \right\rbrack\right) = \frac{\delta}{2}$. We also choose $\varphi$ such that $\Sigma^{\delta,h}$ is $\mathscr{C}^{\infty}$ in the neighbourhood of $\partial R^{\delta,h}$.
What we get is a connected hypersurface by attaching $\Sigma_1^{\delta}$, $\Sigma_2^{\delta}$ and $R^{\delta,h}$:
\[
\Sigma^{\delta,h} = \Sigma_1^{\delta} \cup R^{\delta,h} \cup \Sigma_2^{\delta}
\]
The results of Colette Anné (see \cite{colette_anne} section C.I and \cite{colette_anne_bis}) will allow us to conclude. When $\delta$ tends to $0$, the spectrum of $\Sigma^{\delta,h}$ tends to the spectrum of the disjoint union of $\Sigma_1$, $\Sigma_2$, and the segment $[-h,h]$ with Dirichlet boundary conditions.

The first eivenvalue $\lambda_1(h)$ of the segment $[-h,h]$ for these conditions is equal to $\frac{\pi^2}{4h^2}\cdot$ As a consequence, if we take an integer $N>0$, there exists $h>0$ small enough such that $\lambda_1(h) > \lambda_N(\Sigma_1 \sqcup \Sigma_2)$.
Finally, for all $\epsilon, N > 0$, there exist $\delta>0$ and $h>0$ small enough such that
\[
\left|\lambda_k(\Sigma^{\delta,h})-\lambda_k(\Sigma_1 \sqcup \Sigma_2)\right| < \epsilon \ \ \ \ \ \ \  k=0,\dots,N.
\]

It is also clear that the volume of $\Sigma^{\delta,h}$ converges, when $\delta$ tends to $0$, to the volume of $\Sigma_1 \sqcup \Sigma_2$. Likewise the volume of the interior of $\Sigma^{\delta,h}$ converges to the volume of the interior of $\Sigma_1 \sqcup \Sigma_2$ when $\delta$ tends to $0$.
It shows that $I\left(\Sigma^{\delta,h}\right)$ tends to $I\left(\Sigma_1 \sqcup \Sigma_2\right)$ when $\delta$ tends to $0$.
\end{proof}
Theorem \ref{theoreme_connexe} in the connected case follows immediately.

\begin{center}
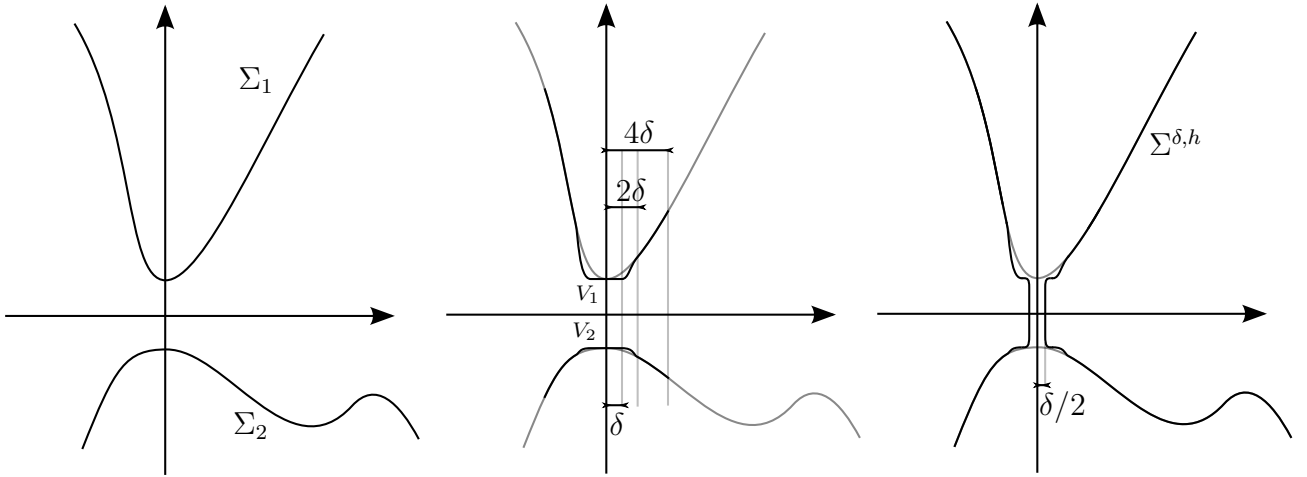

\label{gluing}
\def\svgscale{0.85}
\input{img_recollement_bis.tex}
\captionof{figure}{Gluing the hypersurfaces}
\label{etapes}
\end{center}

\bigskip

\addcontentsline{toc}{chapter}{Bibliographie}
\nocite{*} 
\bibliographystyle{plain}
\bibliography{biblio_article}
\label{lastpage}
\end{document}

%% file: img_spectre_decale_bis.tex
\begingroup%
  \makeatletter%
  \providecommand\color[2][]{%
    \errmessage{(Inkscape) Color is used for the text in Inkscape, but the package 'color.sty' is not loaded}%
    \renewcommand\color[2][]{}%
  }%
  \providecommand\transparent[1]{%
    \errmessage{(Inkscape) Transparency is used (non-zero) for the text in Inkscape, but the package 'transparent.sty' is not loaded}%
    \renewcommand\transparent[1]{}%
  }%
  \providecommand\rotatebox[2]{#2}%
  \ifx\svgwidth\undefined%
    \setlength{\unitlength}{575.37069937bp}%
    \ifx\svgscale\undefined%
      \relax%
    \else%
      \setlength{\unitlength}{\unitlength * \real{\svgscale}}%
    \fi%
  \else%
    \setlength{\unitlength}{\svgwidth}%
  \fi%
  \global\let\svgwidth\undefined%
  \global\let\svgscale\undefined%
  \makeatother%
  \begin{picture}(1,0.15014879)%
    \put(0,0){\includegraphics[width=\unitlength]{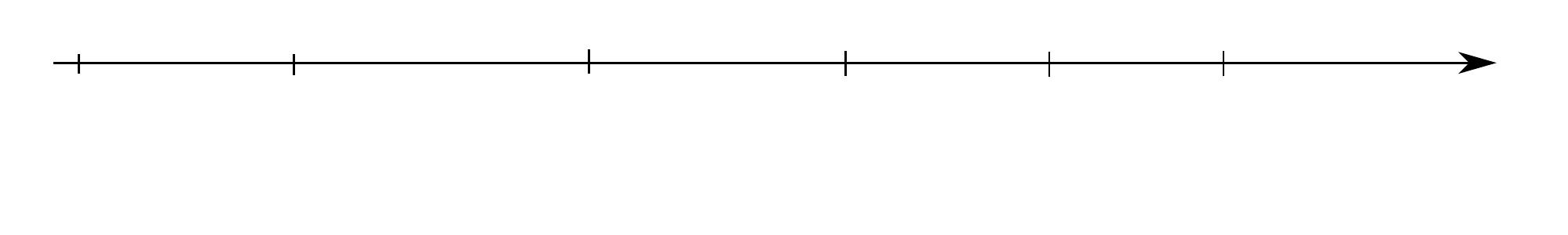}}%
    \put(0.04246584,0.09726561){\color[rgb]{0,0,0}\makebox(0,0)[lt]{\begin{minipage}{0.02260636\unitlength}\raggedright $0$\end{minipage}}}%
    \put(0.00015911,0.0723198){\color[rgb]{0,0,0}\makebox(0,0)[lt]{\begin{minipage}{0.17641954\unitlength}\raggedright $=\lambda_0(\Sii_k)$\end{minipage}}}%
    \put(0.00003927,0.04235039){\color[rgb]{0,0,0}\makebox(0,0)[lt]{\begin{minipage}{0.17666784\unitlength}\raggedright $=\lambda_1(\Sii_k)$\end{minipage}}}%
    \put(0.15089043,0.0914236){\color[rgb]{0,0,0}\makebox(0,0)[lt]{\begin{minipage}{0.17666784\unitlength}\raggedright $\lambda_2(\Sii_k)$\end{minipage}}}%
    \put(0.34222861,0.09072134){\color[rgb]{0,0,0}\makebox(0,0)[lt]{\begin{minipage}{0.17666784\unitlength}\raggedright $\lambda_3(\Sii_k)$\end{minipage}}}%
    \put(0.52534018,0.09288663){\color[rgb]{0,0,0}\makebox(0,0)[lt]{\begin{minipage}{0.04302639\unitlength}\raggedright $\cdots$\end{minipage}}}%
    \put(0.6248644,0.09069311){\color[rgb]{0,0,0}\makebox(0,0)[lt]{\begin{minipage}{0.19633118\unitlength}\raggedright $\lambda_{k+1}(\Sii_k)$\end{minipage}}}%
    \put(0.64401416,0.148755){\color[rgb]{0,0,0}\makebox(0,0)[lt]{\begin{minipage}{0.09304373\unitlength}\raggedright $\lambda_k(\S^n)$\end{minipage}}}%
    \put(0.35127969,0.15015952){\color[rgb]{0,0,0}\makebox(0,0)[lt]{\begin{minipage}{0.10012128\unitlength}\raggedright $\lambda_2(\S^n)$\end{minipage}}}%
    \put(0.1627681,0.1480245){\color[rgb]{0,0,0}\makebox(0,0)[lt]{\begin{minipage}{0.10012128\unitlength}\raggedright $\lambda_1(\S^n)$\end{minipage}}}%
    \put(0.02652926,0.1480245){\color[rgb]{0,0,0}\makebox(0,0)[lt]{\begin{minipage}{0.10012128\unitlength}\raggedright $\lambda_0(\S^n)$\end{minipage}}}%
    \put(0.7570182,0.15073137){\color[rgb]{0,0,0}\makebox(0,0)[lt]{\begin{minipage}{0.25691324\unitlength}\raggedright $\lambda_1(\Sigma) \geqslant \lambda_k(\S^n)+1$\end{minipage}}}%
    \put(-0.00038194,0.01128103){\color[rgb]{0,0,0}\makebox(0,0)[lt]{\begin{minipage}{0.11205973\unitlength}\raggedright $=\lambda_0(\Sigma)$\end{minipage}}}%
  \end{picture}%
\endgroup%

%% file: img_recollement_bis.tex
\begingroup%
  \makeatletter%
  \providecommand\color[2][]{%
    \errmessage{(Inkscape) Color is used for the text in Inkscape, but the package 'color.sty' is not loaded}%
    \renewcommand\color[2][]{}%
  }%
  \providecommand\transparent[1]{%
    \errmessage{(Inkscape) Transparency is used (non-zero) for the text in Inkscape, but the package 'transparent.sty' is not loaded}%
    \renewcommand\transparent[1]{}%
  }%
  \providecommand\rotatebox[2]{#2}%
  \ifx\svgwidth\undefined%
    \setlength{\unitlength}{565.91528444bp}%
    \ifx\svgscale\undefined%
      \relax%
    \else%
      \setlength{\unitlength}{\unitlength * \real{\svgscale}}%
    \fi%
  \else%
    \setlength{\unitlength}{\svgwidth}%
  \fi%
  \global\let\svgwidth\undefined%
  \global\let\svgscale\undefined%
  \makeatother%
  \begin{picture}(1,0.36963903)%
    \put(0,0){\includegraphics[width=\unitlength]{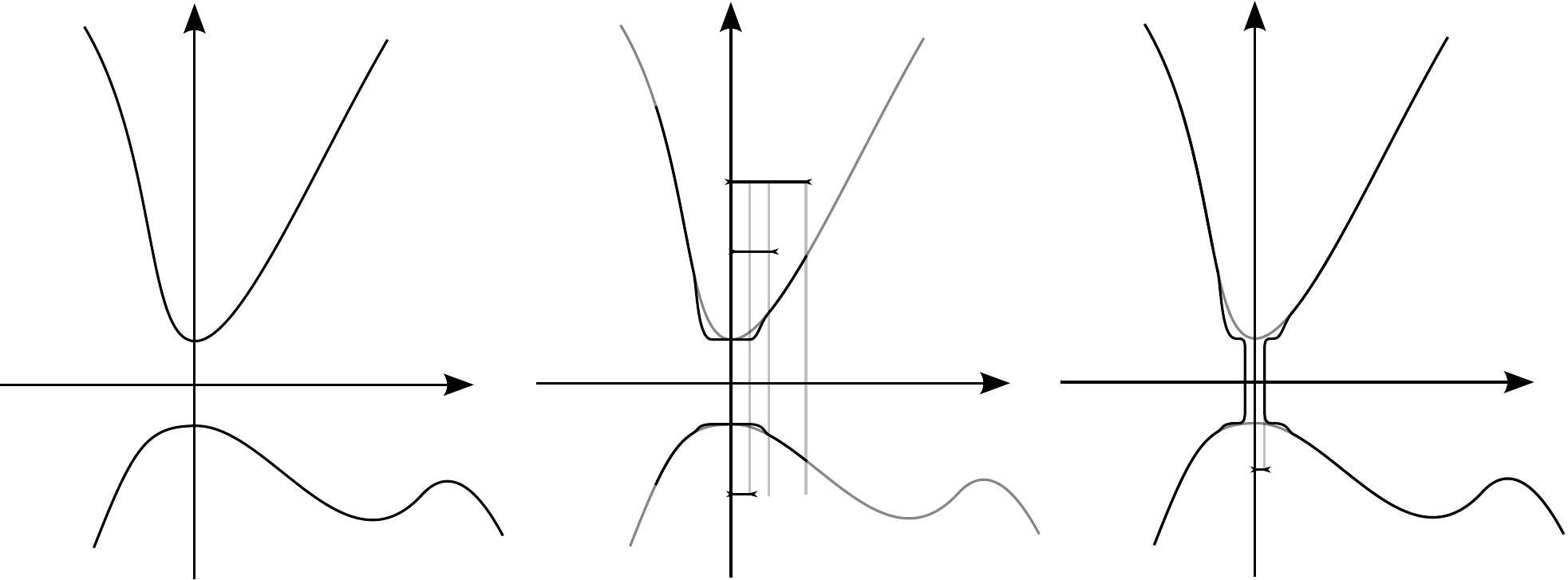}}%
    \put(0.47968479,0.27653635){\color[rgb]{0,0,0}\makebox(0,0)[lt]{\begin{minipage}{0.0497882\unitlength}\raggedright $4\delta$\end{minipage}}}%
    \put(0.47349404,0.23140619){\color[rgb]{0,0,0}\makebox(0,0)[lt]{\begin{minipage}{0.0497882\unitlength}\raggedright $2\delta$\end{minipage}}}%
    \put(0.46867643,0.04866107){\color[rgb]{0,0,0}\makebox(0,0)[lt]{\begin{minipage}{0.05530429\unitlength}\raggedright $\delta$\end{minipage}}}%
    \put(0.180858,0.31918033){\color[rgb]{0,0,0}\makebox(0,0)[lt]{\begin{minipage}{0.09527779\unitlength}\raggedright $\Sigma_1$\end{minipage}}}%
    \put(0.17688123,0.04943204){\color[rgb]{0,0,0}\makebox(0,0)[lt]{\begin{minipage}{0.10424865\unitlength}\raggedright $\Sigma_2$\end{minipage}}}%
    \put(0.88700394,0.27011444){\color[rgb]{0,0,0}\makebox(0,0)[lt]{\begin{minipage}{0.11990323\unitlength}\raggedright $\Sigma^{\delta,h}$\end{minipage}}}%
    \put(0.44238679,0.14823568){\color[rgb]{0,0,0}\makebox(0,0)[lt]{\begin{minipage}{0.08567058\unitlength}\raggedright $_{V_1}$\end{minipage}}}%
    \put(0.43986222,0.11866615){\color[rgb]{0,0,0}\makebox(0,0)[lt]{\begin{minipage}{0.06908892\unitlength}\raggedright $_{V_2}$\end{minipage}}}%
    \put(0.80182326,0.06579355){\color[rgb]{0,0,0}\makebox(0,0)[lt]{\begin{minipage}{0.05530429\unitlength}\raggedright $\delta/2$\end{minipage}}}%
  \end{picture}%
\endgroup%